\documentclass[11pt,leqno]{amsproc}
\usepackage{graphicx}
\usepackage{amsmath,amssymb,amsthm,mathtools}
\usepackage{mathrsfs}
\usepackage{enumitem}
\usepackage{microtype}
\usepackage{hyperref}
\usepackage{url}
\usepackage{amssymb}

\newtheorem{theorem}{Theorem}[section]
\newtheorem{proposition}[theorem]{Proposition}
\newtheorem{lemma}[theorem]{Lemma}
\newtheorem{corollary}[theorem]{Corollary}
\newtheorem{rem}[theorem]{Remark}

\theoremstyle{definition}
\newtheorem{definition}[theorem]{Definition}

\newcommand{\R}{\mathbb R}
\newcommand{\Max}{\operatorname{Max}}
\newcommand{\Sc}{\Sigma}

\newcommand{\dr}{\downarrow \hspace{-3pt}}
\newcommand{\up}{\uparrow \hspace{-3pt}}

\newcommand{\II}{\begin{itemize}}
\newcommand{\I}{\item}
\newcommand{\III}{\end{itemize}}

\begin{document}

\title[The answers to two problems on dcpo models]
{The answers to two problems on dcpo models}

\author[C. Shen]{Shen Chong}
\address[C. Shen]{School of Science,\\ Beijing University of Posts and Telecommunications,\\ Beijing, PR China }
\author[X. Xi]{Xi Xiaoyong}
\address{School of Mathematics and Statistics\\
Jiangsu Normal University\\ Jiangsu China} \email{littlebrook@jsnu.edu.cn}
\author[D. Zhao]{Zhao Dongsheng}
\address[D. Zhao]{Mathematics and Mathematics Education, National Institute of
Education, Nanyang Technological University, 1 Nanyang Walk, Singapore 637616}
\email{dongsheng.zhao@nie.edu.sg}


\date{\today}

\begin{abstract}
A poset model of a topological space $X$ is a poset $P$ such that $X$ is homeomorphic to the maximal point space of $P$ (the set Max($P$) of all maximal points of $P$ equipped with the relative  Scott topology of $P$).
Xi and Zhao proved that if a space has a dcpo model satisfying Lawson condition, it must be coherent and well-filtered.
It is still open wether  every coherent and well filtered space $T_1$ has a dcpo model satisfying  Lawson condition. In this paper, we answer this problem. In another paper, Xi and Zhao proved that every Hausdorff  k-space has a bounded complete dcpo model. It is, however, still unknown whether it is true that if a Hausdorff space is a k-space if it has a bounded complete dcpo model. We will construct a Hausdorff space which is not a k-space but has a bounded complete dcpo model.
\end{abstract}

\maketitle

A poset model of a topological space $X$ is  a poset $P$ such that the subspace $\mbox{Max}(P)$ of all maximal points of $P$ of the  Scott space $\Sigma\!P$ is homeomorphic to $X$. It has been proved by several authors that a topological space has a poset model if and only if it is a $T_1$ space (see \cite{ali}\cite{erne-2011}\cite{zds-2009}).
In \cite{zhao-xi-2016}, it was further proved that every $T_1$ space has a directed complete poset (dcpo, for short) model. Finding poset models with extra properties can help us better understand the topologies of spaces modeled by these posets.
In \cite{zhao-xi-2017}, the topological spaces that have a bounded complete dcpo model are investigated. They proved that every CK-filter defined Hausdorff space (equivalently, k-spaces) has a bounded complete dcpo model. It is still open whether a Hausdorff space is a k-space if it has a bounded complete dcpo model. In \cite{xi-zhao-2017}, it was proved that if a space has a dcpo model satisfying Lawson condition, it must be coherent and well-filtered and it was left open whether the converse is also true: is  every coherent and well-filtered space has a bounded complete dcpo model.

In this paper we give answers to the above two problems. 

\section{Preliminaries}
	Let $(P,\leq)$ be a poset.  A nonempty subset $D\subseteq P$ is \emph{directed}
if for every $d_1,d_2\in D$ there is $d_3\in D$ with
$d_1\leq d_3$ and $d_2\leq d_3$.  A poset is a \emph{dcpo} if every directed
subset has a supremum.  It is \emph{bounded complete} if every nonempty
upper-bounded subset has a supremum.
For $A\subseteq P$, an element $u\in P$ is an upper bound of $A$ if
$a\leq u$ for all $a\in A$.  The supremum $\bigvee A$, when it exists, is
the least upper bound: it is an upper bound of $A$ and is below every other
upper bound.  A subset $U\subseteq P$ is an \emph{upper set} if
$x\in U$ and $x\leq y$ imply $y\in U$.

The \emph{Scott topology} $\Sc(P)$ on a poset $P$ consists of the upper sets
$U$ such that, whenever $D\subseteq P$ is directed and $\bigvee D\in U$,
then $D\cap U\neq\varnothing$.  The set of maximal elements of $P$ is denoted by
$\Max(P)$.  A \emph{dcpo model} of a topological space $X$ is a dcpo $P$ for
which $\Max(P)$, with the subspace topology inherited from $\Sc(P)$, is
homeomorphic to $X$.

\begin{definition} A poset $P$ is said to satisfy the Lawson condition if the restriction of the Scott topology and the Lawson topology on the set of maximal points of $P$ coincide.
\end{definition}

\begin{rem}
\II
\I[(1)] A poset $P$ satisfies Lawson condition iff for any $x\in P$, $(\up\,x)\cap \mbox{Max}(P)$ is a closed set in $\mbox{Max}(P)$, equivalently, $(\up\,x)\cap \mbox{Max}(P)=F\cap \mbox{Max}(P)$ for some Scott closed subset of $P$ \cite{xi-zhao-2017}
\I[(2)] Every bounded complete dcpo satisfies Lawson condition\cite{xi-zhao-2017}. This is equivalent to that every nonempty subset has a infimum if it  has a lower bound.
\I[(3)] For any dcpo $P$, $\dr(\mbox{Max}(P))=P$, that is, every element of $P$ is below some maximal point.
\III
\end{rem}

Throughout, $\omega=\{0,1,2,\ldots\}$, $\omega_1$ is the first uncountable
cardinal, and $\mathfrak c=|\R|=|2^\omega|$ denotes the cardinality of the
continuum.  For a set $M$ and a cardinal $\kappa$, we use the standard notation
\[
\begin{aligned}
	[M]^\kappa&=\{A\subseteq M:|A|=\kappa\},\\
	[M]^{\leq\kappa}&=\{A\subseteq M:|A|\leq\kappa\}.
\end{aligned}
\]
In particular,
\[
\begin{aligned}
	[M]^\omega&=\{A\subseteq M:A\text{ is countably infinite}\},\\
	[M]^{\leq\omega}&=\{A\subseteq M:A\text{ is finite or countably infinite}\},
\end{aligned}
\]
where the empty set and all finite sets belong to $[M]^{\leq\omega}$, and
\[
[M]^{<\omega}=\{A\subseteq M:A\text{ is finite}\}.
\]
Here ``countable'' means cardinality at most $\aleph_0$; thus a countable
set may be finite, while a member of $[M]^\omega$ is required to be
infinite. 

For more about  Scott topology and Lawson topology, we refer the reader to \cite{comp}\cite{Jean-2013}.

\section{A coherent and well-filtered $T_1$ space that does not have a dcpo model satisfying Lawson condition}

By Theorem 1 in  \cite{xi-zhao-2017}, if $X$ has a dcpo model satisfying Lawson condition, then $X$ must be coherent  and well-filtered. However, it is not known whether every coherent and well-filtered space $T_1$  has a dcpo model satisfying Lawson condition\cite{zhao-xi-2017}. In this section we give a coherent, well-filtered and compact $T_1$ space that does not have a dcpo model satisfying Lawson condition.

The following is a general result.

\begin{lemma}\label{space do not have a lawson model}
Let $X$ be a $T_1$ space such that there are two distinct points $a, b\in X$ such that (i) $a$ is not an isolated point; (ii) for any closed set $F$ with $b\not\in F$,  $F$ is finite. Then $X$ does not have a dcpo model satisfying Lawson condition.
\end{lemma}
\begin{proof}We prove by contrapositive. Assume that $X$ has a dcpo model satisfying Lawson condition. We can assume that $X=\mbox{Max}(P)$.

Let $U=\{\in P: (\up x)\cap X =\{x\}$. Since $a\in U$, $U\not=\emptyset$. Clearly, $U$ is an upper set in $P$.
Assume that $D\subseteq P$ is a directed set such that $\bigvee D\in U$. Then $\bigvee D\not\le b$. There is a $d_0\in D$ with $d_0\not\le b$. So the set $F=(\up d_0)\cap X$ is a closed set not containing $b$. By property (ii) of $X$, $F$ is a finite set. Assume that $F=\{a\}\cup\{x_1, x_2, \cdots, x_n\}$. Since $\bigvee D\not\le x_i$ for each $i$, there is a $d_1\in D$ such that $d_0\le d_1$ and $d_1\not\le x_i$ for all $i$ ($i=1, 2, \cdots n$). Then for all $d\ge d_1$, $(\up d)\cap X=\{a\}$, implying that $d_\in U$ for all $d\ge d_1$.
It follows that $U$ is a Scott open set of $P$. But then $U\cap X=\{a\}$, implying that $\{a\}$ is open in $X$, contradicting assumption (i).

 Hence $X$ does not have a dcpo model satisfying Lawson condition.
\end{proof}

Give a set $X$, the co-finite topology $\tau_{cf}$ on $X$ consists of all subsets $U$, either $U=\emptyset$ or $X-U$ is a finite set.

\begin{corollary}
Any infinite set equipped with the  co-finite topology
does not have a dcpo model  satisfying Lawson condition.
\end{corollary}

We now give a $T_1$ space which is coherent and well-filtered and satisfies the properties (i) and (ii) in Lemma \ref{space do not have a lawson model}.

 Define a topology $\tau_0$ on the set $\mathbb{R}$ of all real numbers by taking as its closed sets
\begin{equation}
	\mathcal{C}_0 =\{\R\}\cup[\R]^{<\omega}\cup
	\{F\subseteq \R:0\in F,\ |F|\leq\aleph_0\}.
	\tag{1}
\end{equation}
The family in (1) is closed under arbitrary intersections and finite unions, so it is indeed the closed-set system of a topology. Write $\R_0=(\R,\tau_0)$.

\begin{lemma}[Compact subsets]\label{compact sets}
	For $K\subseteq \R$, $K$ is compact in $\R_0$ if and only if $0\in K$ or $K$ is finite.

In particular, $\R_0$ is a compact space.
\end{lemma}

\begin{proof}
	If $0\in K$, every open cover has a member $U\ni 0$. Its complement is finite, because the only proper closed sets not containing $0$ are finite. Finitely many further members cover the remaining finite set. Finite sets are compact.
	
	Conversely, suppose $0\notin K$ and choose distinct $x_n\in K$. Put
	\[
	U_n=\R\setminus\bigl(\{0\}\cup\{x_m:m\geq n\}\bigr).
	\]
	These are open, cover $K$, and no finite subfamily covers $K$.
\end{proof}

\begin{proposition}\label{ancor countable is coherent}
	The space $\R_0$ is $T_1$, coherent, and well-filtered.
\end{proposition}

\begin{proof}
	Every singleton is finite, hence closed, so $\R_0$ is $T_1$. In a $T_1$ space every subset is saturated. By Lemma\ref{compact sets}, the intersection of two compact subsets either contains $0$, and is therefore compact, or is finite. Thus $\R_0$ is coherent.
	
	Let $\{K_i\}_{i\in I}$ be a filtered family of compact subsets and assume $\bigcap_i K_i\subseteq U$, where $U$ is open. If some $K_i$ omits $0$, it is finite. A filtered family containing a finite member has a member of minimum cardinality, and that member is contained in every member; it is therefore the total intersection and is contained in $U$.
	
	Otherwise every $K_i$ contains $0$, so $0\in U$. The complement $F=\R\setminus U$ is finite. If no $K_i\subseteq U$, then $\{K_i\cap F\}_{i\in I}$ is a filtered family of nonempty finite sets, which has nonempty intersection by the minimum-cardinality argument. This contradicts $F\cap\bigcap_iK_i=\varnothing$. Hence $\R_0$ is well-filtered.
\end{proof}

\begin{theorem}
	The anchored countable space $\R_0$ is $T_1$, compact, coherent, and well-filtered, but it has no bounded-complete dcpo model.
\end{theorem}
\begin{proof}
Note that every open set containing the point $0$ has a finite complement and the point $1$ is not isolated. Also $\R_0$ is compact.
 Now the combination of Lemma\ref{space do not have a lawson model} and Proposition\ref{ancor countable is coherent} deduces this theorem.
\end{proof}

\section{A bounded complete dcpo model for the join of the Euclidean and cocountable topologies}\label{sec:join-euclidean-cocountable}

In this section, we construct a bounded complete dcpo model for the real line equipped with the
join of its usual topology and its cocountable topology. This answers the problem whether every Hausdorff space is a k-space whenever it has a bounded complete dcpo model.

Put
\[
\tau_{\mathrm e}=\text{the usual Euclidean topology on }\R,
\qquad
\tau_{\mathrm{cc}}=\{\varnothing\}\cup
\{\R\setminus A:A\in[\R]^{\leq\omega}\},
\]
and let $\tau_{\vee}=\tau_{\mathrm e}\vee\tau_{\mathrm{cc}}$.

The family
\[
\mathcal B_{\vee}
 =\{U\setminus A:U\in\tau_{\mathrm e},\ A\in[\R]^{\leq\omega}\}
\tag{J1}\label{eq:join-basis}
\]
is a basis for $\tau_{\vee}$.  In fact, every set in (J1) is the intersection
of a Euclidean open set and a cocountable open set, and finite intersections
remain of the same form since
\[
(U\setminus A)\cap(V\setminus B)=(U\cap V)\setminus(A\cup B).
\]

Now let
\[
\mathcal K_{\mathrm e}
 =\{K\subseteq\R:K\neq\varnothing\text{ and }K\text{ is compact in }\tau_{\mathrm e}\}.
\]
For $K\in\mathcal K_{\mathrm e}$ define
\[
\operatorname{Ext}(K)
 =\{A\in[\R]^{\leq\omega}: A\cap K=\varnothing\}.
\]
The second coordinate is a (not necessarily countable) family of admissible
countable labels.  Define
\[
P_{\vee}
 =\{(K,\Lambda):K\in\mathcal K_{\mathrm e},\ \Lambda\subseteq
\operatorname{Ext}(K)\},
\]
with order
\[
(K,\Lambda)\leq(L,\Gamma)
\quad\Longleftrightarrow\quad
L\subseteq K\ \text{ and }\ \Lambda\subseteq\Gamma.
\tag{J2}\label{eq:join-order}
\]
Thus the compact-set coordinate is ordered by reverse inclusion and the label
coordinate by inclusion.

\begin{rem}\label{lem:join-compact-fip}
Let $(K_i)_{i\in I}\subseteq\mathcal K_{\mathrm e}$ and suppose that every
finite subfamily has nonempty intersection.  Then
\[
K^*=\bigcap_{i\in I}K_i\in K_{\mathrm e}
\]
\end{rem}

\begin{theorem}[Bounded completeness]\label{thm:join-bounded-complete}
The poset $P_{\vee}$ is a bounded complete dcpo.  More precisely, whenever
$S\subseteq P_{\vee}$ is nonempty and has an upper bound, then
\[
\bigvee S
 =\left(\bigcap_{(K,\Lambda)\in S}K,
          \bigcup_{(K,\Lambda)\in S}\Lambda\right).
\tag{J3}\label{eq:join-sup}
\]
In particular, (J3) gives the supremum of every directed subset.
\end{theorem}

\begin{proof}
Let $(L,\Gamma)$ be an upper bound of $S$.  By (J2), $L\subseteq K$ for
every $(K,\Lambda)\in S$.  Thus the first coordinates of $S$  have the
finite-intersection property, and their intersection $K^*$ is nonempty and
compact by Remark~\ref{lem:join-compact-fip}.

Put
\[
\Lambda^*=\bigcup_{(K, \Lambda)\in S}\Lambda.
\]
If $A\in\Lambda^*$, then $A$ belongs to a label family $\Lambda$ with
$A\cap K=\varnothing$; since $K^*\subseteq K$, also
$A\cap K^*=\varnothing$.  Hence $\Lambda^*\subseteq \operatorname{Ext}(K^*)$ implying  $(K^*,\Lambda^*)\in P_{\vee}$.

Every member of $S$ is below $(K^*,\Lambda^*)$.  Conversely, if
$(M,\Delta)$ is an upper bound of $S$, then $M\subseteq K$ and
$\Lambda\subseteq\Delta$ for every member of $S$.  Therefore
$M\subseteq K^*$ and $\Lambda^*\subseteq\Delta$, which is exactly
$(K^*,\Lambda^*)\leq(M,\Delta)$.  This proves (J3).  The assertion for
directed sets gives the dcpo property, and the assertion for all nonempty
bounded sets gives bounded completeness.
\end{proof}

\subsection{Maximal elements and basic Scott-open sets}

For $x\in\R$ put
\[
m_x=(\{x\},\operatorname{Ext}(\{x\})).
\]

\begin{proposition}[Maximal elements]\label{prop:join-maximal}
\[
\operatorname{Max}(P_{\vee})=\{m_x:x\in\R\}.
\]
\end{proposition}

\begin{proof}
Given $(K,\Lambda)\in P_{\vee}$ and $x\in K$, every member of $\Lambda$
avoids $K$ and hence avoids $x$.  Therefore
\[
(K,\Lambda)\leq(\{x\},\operatorname{Ext}(\{x\})).
\]
So a maximal element must have singleton first coordinate.  If its label
family is a proper subset of $\operatorname{Ext}(\{x\})$, adjoining a
missing admissible label gives a strictly larger element.  Thus every
maximal element is some $m_x$.

Conversely, if $m_x\leq(L,\Gamma)$, then (J2) gives
$\varnothing\neq L\subseteq\{x\}$, so $L=\{x\}$.  The defining condition
of $P_{\vee}$ gives $\Gamma\subseteq\operatorname{Ext}(\{x\})$, while
(J2) gives the reverse inclusion.  Thus $(L,\Gamma)=m_x$.
\end{proof}

\begin{lemma}[Cocountable generators]\label{lem:join-label-open}
For $A\in[\R]^{\leq\omega}$, let
\[
M_A=\{(K,\Lambda)\in P_{\vee}:A\in\Lambda\}.
\]
Then $M_A$ is Scott-open and
\[
M_A\cap\operatorname{Max}(P_{\vee})=\{m_x:x\notin A\}.
\]
\end{lemma}

\begin{proof}
The set $M_A$ is an upper set.  If the supremum of a directed family lies in
$M_A$, (J3) says that $A$ belongs to the union of its label coordinates,
and hence to one member of the family.  Thus $M_A$ is Scott-open.  Finally,
$A\in\operatorname{Ext}(\{x\})$ if and only if $x\notin A$.
\end{proof}

\begin{lemma}[Euclidean generators]\label{lem:join-euclidean-open}
For $U\in\tau_{\mathrm e}$, let
\[
G_U=\{(K, \Lambda)\in P_{\vee}:K\subseteq U\}.
\]
Then $G_U$ is Scott-open and
\[
G_U\cap\operatorname{Max}(P_{\vee})=\{m_x:x\in U\}.
\]
\end{lemma}

\begin{proof}
The set $G_U$ is an upper set.  Let $D=\{(K_i,\Lambda_i):i\in I\}$ be
directed and suppose that its supremum $(K^*,\Lambda^*)$ belongs to $G_U$.
If no $K_i$ were contained in $U$, put $F=\R\setminus U$.  Then every
$K_i\cap F$ is nonempty.  Fix $i_0$.  For every finite collection of
indices, directedness provides $j$ with $K_j$ contained in all the
corresponding compact sets; because $K_j\not\subseteq U$, it meets $F$.
Thus the closed sets $K_{i_0}\cap K_i\cap F$ have the finite-intersection
property in $K_{i_0}$.  Compactness gives $K^*\cap F\neq\varnothing$,
contradicting $K^*\subseteq U$.  Some $K_i$ is therefore contained in $U$, 
hence $(K_i, \Lambda_i)\in G_U$. So $G_U$ is Scott open.  That $G_U\cap\operatorname{Max}(P_{\vee})=\{m_x:x\in U\}$ is trivial.
\end{proof}

\subsection{The maximal Scott topology}

\begin{theorem}[Join-topology trace theorem]\label{thm:join-trace}
Under the bijection $x\mapsto m_x$, the subspace Scott topology on
$\operatorname{Max}(P_{\vee})$ is precisely
$\tau_{\mathrm e}\vee\tau_{\mathrm{cc}}$.
\end{theorem}

\begin{proof}
Lemmas~\ref{lem:join-label-open} and
\ref{lem:join-euclidean-open} show that the relative Scott topology contains
both $\tau_{\mathrm{cc}}$ and $\tau_{\mathrm e}$, hence contains their join.
For the converse, let $O$ be Scott-open and put
\[
V_O=\{x\in\R:m_x\in O\}.
\]
If $V_O=\varnothing$, it is open.  Assume $x\in V_O$ and set
$B=\R\setminus V_O$.  The finite-label histories
\[
h_\Lambda=(\{x\},\Lambda),
\qquad \Lambda\in[\operatorname{Ext}(\{x\})]^{<\omega},
\]
form a directed family with supremum $m_x$.  Hence there is a finite
$\Lambda$ with $h_\Lambda\in O$.  
Put $C_\Lambda=\bigcup\Lambda=\bigcup\{K: K\in\Lambda\}$; this is
countable and $x\notin C_\Lambda$.

We claim that there exists an  Euclidean neighborhood $W$ of $x$, such that  $B\cap W$
is countable.  

In fact, if not, one can choose distinct points
\[
b_n\in B\cap\left(x-\frac1{n+1},x+\frac1{n+1}\right)\setminus C_\Lambda
\qquad(n<\omega).
\]
Then $b_n\to x$.  Let
\[
K_n=\{x\}\cup\{b_m: m\geq n\}.
\]
Each $K_n$ is Euclidean compact, the sequence is decreasing, and
$\bigcap_nK_n=\{x\}$.  Also $K_n\cap C_\Lambda=\varnothing$, so
$d_n=(K_n,\Lambda)\in P_{\vee}$ and
$\bigvee_nd_n=h_\Lambda$ by (J3).  Scott inaccessibility gives $d_n\in O$
for some $n$.  Since $b_n\in K_n$ and every label in $\Lambda$ avoids
$K_n$, we have $(K_n,\Lambda)\leq m_{b_n}$.  Upperness of $O$ then gives
$m_{b_n}\in O$, so $b_n\in V_O$,   contradicting $b_n\in B=\R \ V_O$.

Thus there is an Euclidean open neighbourhood $W$ of $x$ such that  $A=B\cap W$ is countable and
\[
x\in W\setminus A\subseteq =W\setminus B=W\setminus (\R\setminus V_O)\subseteq V_O.
\]
By the basis (J1), $V_O\in\tau_{\vee}$.

The proof is completed.
\end{proof}

The combination of Theorem~\ref{thm:join-bounded-complete},
Proposition~\ref{prop:join-maximal}, and
Theorem~\ref{thm:join-trace} deduces the following result.

\begin{proposition}[Bounded complete model for the join]\label{thm:join-main}
The space
\[
\bigl(\R,\tau_{\mathrm e}\vee\tau_{\mathrm{cc}}\bigr)
\]
has a bounded complete dcpo model.  More precisely, $P_{\vee}$ is a bounded
complete dcpo and $x\mapsto m_x$ is a homeomorphism from this space onto
the maximal-point Scott space of $P_{\vee}$.
\end{proposition}

Note that the compact subsets of $\bigl(\R,\tau_{\mathrm e}\vee\tau_{\mathrm{cc}}\bigr)$ are finite sets and the space is Hausdorff, hence it is not a k-space. Thus we have the following.

\begin{theorem}
There is a Hausdorff space that has a bounded complete dcpo model and the space is not a k-space. 
\end{theorem}

\end{document}